\documentclass[a4paper,10pt]{siamltex}
\usepackage[english]{babel}
\usepackage{graphicx,pgf}
\usepackage{enumitem}
\usepackage{hyperref}
\usepackage{url}
\usepackage[utf8x]{inputenc}
\usepackage{amsmath}
\usepackage{amsfonts}
\usepackage{todonotes,comment}
\usepackage{mathabx}
\usepackage{listings}

\newtheorem{remark}[theorem]{Remark}
\newtheorem{example}[theorem]{Example}
\newtheorem{assumption}{Assumption}

\newcommand{\mpol}[1]{\mathbb C[x]^{#1\times #1}}
\newcommand{\mc}[1]{#1}
\title{On a class of matrix pencils and $\ell$-ifications equivalent to a given matrix polynomial\thanks{Work supported by Gruppo Nazionale di Calcolo Scientifico (GNCS) of INdAM}}
\author{Dario A. Bini
\thanks{Dipartimento di Matematica, Universit\`a di Pisa,
Largo Bruno Pontecorvo 5, 56127 Pisa ({\tt bini@dm.unipi.it})}
\and
Leonardo Robol
\thanks{Scuola Normale Superiore, Piazza dei Cavalieri, Pisa
({\tt leonardo.robol@sns.it})}
}

\begin{document}
  \maketitle

\begin{abstract}
A new class of linearizations and $\ell$-ifications for $m\times m$
matrix polynomials $P(x)$ of degree $n$ is proposed.  The
$\ell$-ifications in this class have the form $A(x) = D(x) + (e\otimes I_m) W(x)$
where $D$ is a block diagonal matrix polynomial with blocks $B_i(x)$
of size $m$, $W$ is an $m\times qm$ matrix polynomial and $e=(1,\ldots,1)^t\in\mathbb C^q$,
for a suitable integer $q$. 
The blocks
$B_i(x)$ can be chosen a priori, subjected to some restrictions. Under
additional assumptions on the blocks $B_i(x)$ the matrix polynomial
$A(x)$ is a strong $\ell$-ification, i.e., the reversed polynomial of
$A(x)$ defined by $A^\#(x) := x^{\deg A(x)} A(x^{-1})$ is an
$\ell$-ification of $P^\#(x)$.  The eigenvectors of the matrix
polynomials $P(x)$ and $A(x)$ are related by means of explicit
formulas.  Some practical examples of $\ell$-ifications are provided.
A strategy for choosing  $B_i(x)$ in such a way that  $A(x)$ is
a well conditioned linearization of $P(x)$ is proposed.  Some
numerical experiments  that validate the theoretical
results are reported.

  \end{abstract}


\section{Introduction}
A standard way to deal with an $m\times m$ matrix polynomial $P(x) =
\sum_{i = 0}^n P_i x^i$ is to convert it to a linear pencil, that is
to a linear matrix polynomial of the form $L(x) = Ax - B$ where $A$
and $B$ are $mn \times mn$ matrices such that $\det P(x) = \det L(x)$.
This process, known as {\em linearization}, has been considered
in \cite{glr:book}.

In certain cases, like for matrix polynomials modeling Non-Skip-Free
stochastic processes \cite{blm:book}, it is more convenient to reduce
the matrix polynomial 
to a quadratic polynomial of the form $Ax^2+Bx+C$, where $A,B,C$ are
 matrices of suitable size \cite{blm:book}. The process that we obtain
this way is referred to as {\em quadratization}.  If $P(x)$ is a
matrix power series, like in M/G/1 Markov chains
\cite{neuts:book2,neuts:book1}, the quadratization of $P(x)$ can be
obtained with block coefficients of infinite size \cite{ram}.  In this
framework, the quadratic form is desirable since it is better suited
for an effective solution of the stochastic model; in fact it
corresponds to a QBD process for which there exist efficient solution
algorithms \cite{blm:book}, \cite{lr:book}. In other situations it is
preferable to reduce the matrix polynomial $P(x)$ of degree $n$ to a
matrix polynomial of lower degree $\ell$. This process is called
 {\em $\ell$-ification} in \cite{ellification}.

Techniques for linearizing a matrix polynomial have been widely
investigated. Different companion forms of a matrix polynomial have
been introduced and analyzed, see for instance \cite{acl,ddm12,mmmm06}
and the literature cited therein. A wide literature exists on matrix
polynomials with contribution of many authors
\cite{at,coreless2004,ddm9,ddm10,ddm12,glr:book,hmmt,hmt06,hmt09, ttz,
  tz}, motivated both by the theoretical interest of this subject and
by the many applications that matrix polynomials have
\cite{blm:book,lr:book,mmmm-vib,neuts:book2,neuts:book2,tm}.  Techniques for
reducing a matrix polynomial, or a matrix power series into quadratic
form, possibly with coefficients of infinite size, have been
investigated in \cite{blm:book,ram}. Reducing a matrix polynomial to a
polynomial of degree $\ell$ is analyzed in \cite{ellification}.

Denote by $\mpol m$ the set of $m\times m$ matrix polynomials over the
complex field $\mathbb C$.  If $P(x)=\sum_{i=0}^nP_ix^i\in\mpol m$ and
$P_n\ne 0$ we say that $P(x)$ has \emph{degree} $n$.  If $\det P(x)$ is not
identically zero we say that $P(x)$ is \emph{regular}. Throughout the
paper we assume that $P(x)$ is a regular polynomial of degree $n$. The
following definition is useful in our framework.

\begin{definition}\label{def:1}
  Let $P(x)\in \mpol m$ be a matrix polynomial of degree $n$. Let $q$
  be an integer such that $0 < q \leq n$. We say that a matrix
  polynomial $Q(x) \in \mpol{qm}$ is \emph{equivalent} to $P(x)$, and
  we write $P(x)\approx Q(x)$ if there exist two matrix polynomials
  $E(x),F(x)\in\mpol{qm}$ such that $\det E(x)$ and $\det F(x)$ are
  nonzero constants, that is $E(x)$ and $F(x)$ are {\em unimodular},
  and
  \[
    E(x) Q(x) F(x) = \begin{bmatrix} 
     I_{m(q-1)}&0 \\
    0  &   P(x) \\
	\end{bmatrix}=:I_{m(q-1)}\oplus P(x).
  \]
Denote $P^\#(x)=x^nP(x^{-1})$ the {\em reversed polynomial}
obtained by reverting the order of the coefficients. We say that the
polynomials $P(x)$ and $Q(x)$ are {\em strongly equivalent} if
$P(x)\approx Q(x)$ and $P^\#(x) \approx Q^\#(x)$.  If the degree of
$Q(x)$ is 1 and $P(x)\approx Q(x)$  we say that $Q(x)$ is a
{\em linearization} of $P(x)$. Similarly, we say that $Q(x)$ is a
{\em strong linearization} if $Q(x)$ is strongly equivalent to $P(x)$ and
$\deg Q(x)=1$. If $Q(x)$ has degree $\ell$ we use the terms 
\emph{$\ell$-ification} and \emph{strong $\ell$-ification}.
\end{definition}

It is clear from the definition
that $P(x) \approx Q(x)$ implies $\det P(x) = \kappa \det Q(x)$ where $\kappa$
is some nonzero constant, but the
converse is not generally true. The equivalence property is actually
stronger because it preserves also the eigenstructure of the matrix
polynomial, and not only the eigenvalues. For a more in-depth view of
this subject see \cite{ellification}.

In the literature, a number of different linearizations have been
proposed.  The most known are probably the Frobenius and the Fiedler 
linearizations \cite{ddm10}.
One of them is, for example,

 \begin{equation}\label{eq:frob}
xA-B=x    \begin{bmatrix}
   I_m &        &     &  \\
       & \ddots &     &  \\
       &        & I_m &  \\
       &        &     & P_{n} \\
   \end{bmatrix}
-   \begin{bmatrix}
         &        &     & -P_0 \\
     I_m &        &     & -P_1 \\
         & \ddots &     & \vdots \\
         &        & I_m & -P_{n-1} \\
   \end{bmatrix},
 \end{equation}
 where $I_m$ denotes the identity matrix of size $m$.

\subsection{New contribution}
In this paper we provide a general way to transform a given $m\times
m$ matrix polynomial $P(x)=\sum_{i=0}^nP_i x^i$ of degree $n$ into a
strongly equivalent matrix polynomial $A(x)$ of lower degree $\ell$
and larger size endowed with a strong 
structure. The
technique relies on representing $P(x)$ with respect to a basis of
matrix polynomials of the form $C_i(x)=\prod_{j=1,\, j\ne i}^q
B_j(x)$, $i=1,\ldots, q$, where $B_i(x)$ such that
$\deg B_i(x) = d_i$ satisfy the following requirements:
\begin{enumerate}
  \item For every $i,j$, $B_i(x)B_j(x) - B_j(x)B_i(x) = 0$, i.e., 
  the $B_i(x)$ commute;
  \item $B_i(x)$ and $B_j(x)$ are right coprime for every $i \neq j$. This implies that there exist $\alpha_{i,j}(x)$, $\beta_{i,j}(x)$ appropriate matrix polynomials such that
  $B_i(x) \alpha_{i,j}(x) + B_j(x) \beta_{i,j}(x) = I$. 
\end{enumerate}

The above conditions are sufficient to obtain an $\ell$-ification. In order to provide a strong $\ell$-ification, we need the following additional assumptions
\begin{enumerate}
\item $\deg B_i(x)=d$ for $i=1,\ldots,q$;
\item $B_i^\#(x)$ and $B_j^\#(x)$ are right coprime for every $i \neq j$.
\end{enumerate}

According to the choice of the basis we arrive at different
$\ell$-ifications $A(x)$, where $\ell \ge \lceil n/q\rceil$ is
determined by the degree of the $B_i(x)$, represented as a $q\times q$
block diagonal matrix with $m\times m$ blocks plus a matrix of rank at
most $m$. 

Moreover, we provide an explicit version of right and left
eigenvectors of $A(x)$ in the general case.

An example of $\ell$-ification  $A(x)$ is given by
\[\begin{split}
&A(x)=\mc  D(x)+(e\otimes I_m)[W_1,\ldots,W_q],\\
&\mc D(x)=\diag(B_1(x),\ldots,B_q(x)),\quad 
   e=(1,\ldots,1)^t\in\mathbb C^q,\\
&B_i(x)=b_i(x)I_m,~~\hbox{for } i=1,\ldots,q-1,\\
&B_q(x)=b_q(x)P_n+s I_m,~~ d_q=\deg b_q(x),\\
&
W_i(x)\in\mathbb C[x]^{m\times m},\quad \deg W_i(x)<\deg B_i(x),
\end{split}
\]
where $b_1(x),\ldots, b_q(x)$ are pairwise co-prime monic polynomials of
degree $d_1,\ldots,d_q$, respectively, such that $n=d_1+\cdots
+d_q$, and  $s$ is such that $\lambda b_q(\xi) + s \neq 0$  
for any eigenvalue $\lambda$ of $P_n$ and for any root $\xi$ of $b_i(x)$
for $i=1,\ldots,q-1$.
The matrix polynomial $A(x)$ has degree $\ell= \max \{ d_1, \dots, d_q \} \ge \lceil \frac n q \rceil$ and
size $mq\times mq$. 

If $b_i(x)=x-\beta_i$ are linear polynomials
then $\deg A(x)=1$ and the above equivalence turns into a
strong linearization, moreover the eigenvalues of $P(x)$ can be viewed
as the generalized eigenvalues of the matrix pencil $A(x)$
{\small
\[
A(x) = x\begin{bmatrix}
 I\\
 & \ddots\\
& & I\\
&&&P_n
\end{bmatrix}
-
\begin{bmatrix}
\beta_1 I\\
 & \ddots\\
& &\beta_{n-1} I\\
&&& \beta_n P_n - s I
\end{bmatrix}+
\begin{bmatrix}I\\ \vdots \\ \vdots\\ I\end{bmatrix}
[W_1 \ldots W_n]
\]
}
where 
\begin{equation}\label{eq:introW}
 W_i = \begin{cases}
   \frac{P(\beta_i)} {\prod_{j=1,\, j\neq i} ^ {n-1} (\beta_i - \beta_j )}( 
     (\beta_i-\beta_n) P_n + s I_m)^{-1}  & \text{for } i < n,\\[2ex]
 \frac{ P(\beta_n)}{\prod_{j=1}^{n-1}(\beta_n-\beta_j)} - sI_m - s \sum_{j = 1}^{n-1} \frac{W_i}{\beta_n - \beta_j} & \text{otherwise}. \\
 \end{cases}
\end{equation}

If $P(x)$ is a scalar polynomial then $\det A(x)=\prod_{i=1}^n
(x-\beta_i)(\sum_{j=1}^n \frac{W_j}{x-\beta_j}+1)$ so that the
eigenvalue problem can be rephrased in terms of the secular equation
$\sum_{j=1}^n \frac{W_j}{x-\beta_j}+1=0$. Motivated by this fact, we 
will refer to this linearization as {\em secular linearization}.

Observe that this kind of linearization relies on the representation
of $P(x)-\prod_{i=1}^nB_i(x)$ in the Lagrange basis formed by
$C_i(x)=\prod_{j=1,\,j\ne i}^nB_j(x)$, $i=1,\ldots,n$ which is
different from the linearization given in \cite{acl} where the pencil
$A(x)$ has an arrowhead structure. Unlike the linearization of
\cite{acl}, our linearization does not introduce eigenvalues at
infinity.

This secular linearization has some advantages with respect to the
Frobenius linearization \eqref{eq:frob}.  For a monic matrix polynomial we show that with the
linearization obtained by choosing $\beta_i=\omega_n^i$, where
$\omega_n$ is a principal $n$th root of 1, our linearization
is
unitarily similar to the block Frobenius pencil associated with
$P(x)$. By choosing $\beta_i=\alpha\omega_n^i$, we obtain a pencil
unitarily similar to the scaled Frobenius one. With these choices,
the eigenvalues of the secular linearization have the same condition number as
the eigenvalues of the (scaled) Frobenius matrix.

This observation leads to better choices of the nodes $\beta_i$
performed according to the magnitude of the eigenvalues of $P(x)$. In
fact, by using the information provided by the tropical roots in the
sense of \cite{bns}, we may compute at a low cost particular values of
the nodes $\beta_i$ which greatly improve the condition number of the
eigenvalues.  From an experimental analysis we find that in most cases
the conditioning of the eigenvalues of the linearization obtained this
way is lower by several orders of magnitude with respect to the
conditioning of the eigenvalues of the Frobenius matrix even if it is
scaled with the optimal parameter.

Our experiments, reported in Section~\ref{sec:app} are based on some randomly generated polynomials and
on some problems taken from the repository NLEVP \cite{nlevp}.

We believe that the information about the tropical roots, used in
\cite{gs09} for providing better numerically conditioned problems, can
be more effectively used with our $\ell$-ification. This analysis is part of
our future work.

Another advantage of this representation is that any matrix in the
form ``diagonal plus low-rank'' can be reduced to 
Hessenberg form $H$ by means of Givens rotation with a low number of
arithmetic operations provided that the diagonal is real. Moreover,
the function $p(x)=\det(xI-H)$ as well as the Newton correction 
$p(x)/p'(x)$ can be computed
in $O(nm^2)$ operations \cite{inprep}. This fact can be used to implement the Aberth
iteration in $O(n^2m^3)$ ops instead of $O(nm^4+n^2m^3)$ of
\cite{noferini}.  This complexity bound seems optimal in the sense
that for each one of the $mn$ eigenvalues all the $m^2(n+1)$ data are
used at least once. 



The paper is organized as follows. In Section \ref{sec:different} we
provide the reduction of any matrix polynomial $P(x)$ to the
equivalent form 
\[
A(x)=\mc D(x)+(e\otimes I_m)[W_1(x),\ldots,W_q(x)],
\]
that is, the $\ell$-ification of $P(x)$.  In Section \ref{sec:strong}
we show that $P(x)$ is strongly equivalent to $A(x)$ in the sense of
Definition \ref{def:1}. In Section \ref{sec:eig} we provide the
explicit form of left and right eigenvectors of $A(x)$ relating them to the corresponding eigenvectors of $P(x)$. In Section
\ref{sec:particular} we analyze the case where $B_i(x)=b_i(x)I$ for
$i=1,\ldots,q-1$ and $B_q(x)=b_q(x)P_n+sI$ for scalar polynomials
$b_i(x)$. Section \ref{sec:comput} outlines an algorithm for computing
the $\ell$-ifications. Finally, in Section \ref{sec:app} we present
the results of some numerical experiments.

\section{A diagonal plus low rank $\ell$-ification}\label{sec:different}
Here we recall a known companion-like matrix for scalar polynomials
represented as a diagonal plus a rank-one matrix, provide a more
general formulation and then extend it to the case of matrix
polynomials. This form was introduced by B.T.~Smith in \cite{smith},
as a tool for providing inclusion regions for the zeros of a
polynomial, and used by G.H.~Golub in \cite{golub} in the analysis of
modified eigenvalue problems.

Let $p(x)=\sum_{i=0}^np_ix^i$ be a polynomial of degree $n$ with
complex coefficients, assume $p(x)$ monic, i.e., $p_n=1$, consider a
set of pairwise different complex numbers $\beta_1,\ldots,\beta_n$ and set
$e=(1,\ldots,1)^t$.  Then it holds that \cite{golub}, \cite{smith}

\begin{equation}\label{eq:secularlin}
\begin{split}
&p(x)=\det(xI-D+ew^t),\\
&D=\diag(\beta_1,\ldots,\beta_n),\quad  
  w=(w_i),~w_i = \frac{p(\beta_i)}{\prod_{j \neq i} (\beta_i - \beta_j)}.
\end{split}
\end{equation}

Now consider a monic polynomial $b(x)$ of degree $n$ factored as
$b(x)=\prod_{i=1}^q b_i(x)$, where $b_i(x)$, $i=1,\ldots, q$ are monic
polynomials of degree $d_i$ which are co-prime, that is,
$\gcd(b_i(x),b_j(x))=1$ for $i\ne j$, where $\gcd$ denotes the monic
greatest common divisor. Recall that given a pair $u(x),v(x)$ of
polynomials there exist unique polynomials $s(x),r(x)$ such that $\deg
s(x)<\deg v(x)$, $\deg r(x)<\deg u(x)$, and
$u(x)s(x)+v(x)r(x)=\gcd(u(x),v(x))$. From this property it follows
that if $u(x)$ and $v(x)$ are co-prime, there exists $s(x)$
such that $s(x)u(x)\equiv 1\mod v(x)$.  This polynomial can be viewed as the
reciprocal of $u(x)$ modulo $v(x)$. Here and hereafter we denote
$u(x)\mod v(x)$ the remainder of the division of $u(x)$ by $v(x)$.

This way, we may uniquely represent any polynomial of degree $n$ in terms of 
the generalized Lagrange polynomials $c_i(x)=b(x)/b_i(x)$,
$i=1,\ldots,q$ as follows.

\begin{lemma}\label{lem:1} Let $b_i(x)$, $i=1,\ldots, q$ be co-prime monic
polynomials such that $\deg b_i(x)=d_i$ and $b(x)=\prod_{i=1}^q
b_i(x)$ has degree $n$. Define $c_i(x)=b(x)/b_i(x)$.  
Then there exist polynomials $s_i(x)$
such that $s_i(x)c_i(x)\equiv 1\mod b_i(x)$, moreover, any monic polynomial
$p(x)$ of degree $n$ can be uniquely written as 
\begin{equation}\label{eq:1}
\begin{split}
&p(x)=b(x)+\sum_{i=1}^q
w_i(x)c_i(x),\\
& w_i(x)\equiv p(x)s_i(x)\mod b_i(x),\quad i=1,\ldots, q,
\end{split}
\end{equation}
where $\deg w_i(x)<d_i$.
\end{lemma}
\begin{proof}
Since $\gcd(b_i(x),b_j(x))=1$ for $i\ne j$ then $b_i(x)$ and
$c_i(x)=\prod_{j\ne i}b_j(x)$ are co-prime.  Therefore there exists
$s_{i}(x)\equiv 1/c_i(x)\mod b_i(x)$. Moreover, setting
$w_i(x)\equiv p(x)s_i(x)\mod b_i(x)$ for $i=1,\ldots, q$, it turns out that
the equation $p(x)=b(x)+\sum_{i=1}^q w_i(x)c_i(x)$ is satisfied modulo
$b_i(x)$ for $i=1,\ldots,q$. For the primality of
$b_1(x),\ldots,b_q(x)$, this means that the polynomial
$\psi(x):=p(x)-b(x)-\sum_{i=1}^q w_i(x)c_i(x)$ is a multiple of
$\prod_{i=1}^q b_i(x)$ which has degree $n$. Since $\psi(x)$ has
degree at most $n-1$ it follows that $\psi(x)=0$. That is \eqref{eq:1}
provides a representation of $p(x)$. This representation is unique
since another representation, say,  given by $\tilde w_i(x)$, $i=1,\ldots,q$,
would be such that $\sum_{i=1}^q(\tilde w_i(x)- w_i(x))c_i(x)=0$,
whence $(\tilde w_i(x)- w_i(x))c_i(x)\equiv 0\mod b_i(x)$. That is, for the
co-primality of $b_i(x)$ and $c_i(x)$, the polynomial $\tilde
w_i(x)-w_i(x)$ would be multiple of $b_i(x)$.  The property
$\deg(b_i(x))<\deg(\tilde w_i(x)-w_i(x))$ implies that $\tilde
w_i(x)-w_i(x)=0$.
\end{proof}

The polynomial $p(x)$ in Lemma \ref{lem:1} can be represented by means
of the determinant of a (not necessarily linear) matrix polynomial as
expressed by the following result which provides a generalization of
\eqref{eq:secularlin}

\begin{theorem}\label{lem:2}
Under the assumptions of Lemma \ref{lem:1} we have
\[
p(x)=\det A(x),\quad A(x)=D(x)+e[w_1(x),\ldots,w_q(x)]
\]
for $D=\diag(b_1(x),\ldots,b_q(x))$ and $e=[1,\ldots,1]^t$.
\end{theorem}
\begin{proof}
Formally, one has $A(x)=D(x)(I+D(x)^{-1}e[w_1(x),\ldots,w_q(x)])$ so
that 
\[
\begin{split}\det A(x)&=\det
D(x)\det(I_{q}+D(x)^{-1}e[w_1(x),\ldots,w_q(x)])\\
&=b(x)(1+
[w_1(x),\ldots,w_q(x)]D(x)^{-1}e),
\end{split}
\]
 where $b(x)=\prod_{i=1}^q b_i(x)$.
Whence, we find that $\det A(x)=b(x)+\sum_{i=1}^q w_i(x)c_i(x)=p(x)$,
where the latter equality holds in view of Lemma \ref{lem:1}.
\end{proof}

Observe that for $d_i=1$ the above result reduces to
\eqref{eq:secularlin} where $w_i$ are constant polynomials.  From the
computational point of view, the polynomials $w_i(x)$ are obtained by
performing a polynomial division since $w_i(x)$ is the remainder of
the division of $p(x)s_i(x)$ by $b_i(x)$.  

\subsection{Strong $\ell$-ifications for matrix polynomials} \label{sec:strong}
The following technical Lemma is needed to prove the next Theorem
\ref{thm:3}.

\begin{lemma} \label{lem:unimodular_f1}
  Let $B_1(x), B_2(x) \in \mpol m$ be
  regular and such that
  $B_1(x)B_2(x) = B_2(x)B_1(x)$. Assume that $B_1(x)$ and $B_2(x)$ are right co-prime, that is,
  there exist $\alpha(x), \beta(x) \in \mpol m$ such that
  $
    B_1(x) \alpha(x) + B_2(x) \beta(x) = I_m.
  $
  Then the $2 \times 2$ block-matrix polynomial
  $
    F(x) = \left[\begin{smallmatrix}
      \alpha(x) & B_2(x) \\
      -\beta(x) & B_1(x) \\
    \end{smallmatrix}\right]
  $
  is unimodular. 
\end{lemma}

\begin{proof} 
From the decomposition
  \[
    \begin{bmatrix}
      I_m & 0 \\
      B_1(x) & -B_2(x) \\
    \end{bmatrix}
    \begin{bmatrix} 
    \alpha(x) &  B_2(x)  \\
    -\beta(x) &   B_1(x) \\
    \end{bmatrix} = 
    \begin{bmatrix}
       \alpha(x) &B_2(x) \\
       I_m  & 0\\
    \end{bmatrix}
  \]
 we have $-\det B_2(x)\det F(x)=-\det B_2(x)$. Since $B_2(x)$ is regular then $\det F(x)=1$.
\end{proof}

\begin{lemma}\label{lem:primeproduct}
  Let $P(x), Q(x)$ and $T(x)$ be pairwise commuting right co-prime matrix
  polynomials. Then $P(x)Q(x)$ and $T(x)$ are also right
  co-prime. 
\end{lemma}

\begin{proof}
  We know that there exists $\alpha_P(x), \beta_P(x), \alpha_Q(x), \beta_Q(x)$, matrix polynomials such that
  \[
    P(x)\alpha_P(x) + T(x)\beta_P(x) = I, \quad
    Q(x)\alpha_Q(x) + T(x)\beta_Q(x) = I. 
  \]
  We shall prove that there exist appropriate $\alpha(x), \beta(x)$
  matrix polynomials
  such that $P(x)Q(x)\alpha(x) + T(x)\beta(x) = I$. We have
  \begin{align*}
    &P(x)Q(x) ( \alpha_Q(x) \alpha_P(x) ) + 
    T(x) (P(x) \beta_Q(x) \alpha_P(x)+ \beta_P(x)) = \\
    &P(x) ( Q(x)\alpha_Q(x) + T(x) \beta_Q(x) ) \alpha_P(x) + 
    T(x) \beta_P(x) = \\
    &P(x) \alpha_P(x) + T(x) \beta_P(x) = I,
  \end{align*}
  where the first equality holds since $T(x)P(x) = P(x)T(x)$. 
  So we can conclude that also $P(x)Q(x)$ and $T(x)$ are
  right coprime, and a possible choice for $\alpha(x)$ and $\beta(x)$
  is:
  \[
    \alpha(x) = \alpha_Q(x) \alpha_P(x), \qquad
    \beta(x) = P(x)\beta_Q(x) \alpha_P(x) + \beta_P(x). 
  \]
\end{proof}

Now we are ready to prove the main result of this section, 
which provides an $\ell$-ification of a matrix polynomial
$P(x)$ which is not generally strong. Conditions under
which this $\ell$-ification is strong are given
in Theorem~\ref{thm:stronglin}.

\begin{theorem}\label{thm:3}
  Let $P(x) = \sum_{i = 0}^n P_i x^i$, $B_1(x),
\ldots, B_q(x)$, and $W_1(x), \ldots, W_q(x)$ be polynomials in
$\mpol m$. 
Let  $C_i(x) = \prod_{\substack{j = 1,\,}{j \neq i}}^q
B_j(x)$ and suppose that the following conditions hold:
\begin{enumerate}
 \item \label{cond:1} 
    $P(x) = \prod_{i = 1}^q B_i(x) + \sum_{i = 1}^q W_i(x) C_i(x)$; 
 \item \label{cond:2} the polynomials $B_i(x)$ are regular, commute, i.e., $B_i(x)B_j(x) -
   B_j(x)B_i(x) = 0$ for any $i,j$, and are pairwise
   right co-prime.
\end{enumerate}
Then the matrix polynomial $A(x)$ defined as 
\begin{equation}\label{eq:A(x)}
 A(x) = D(x) + (e\otimes I_m) [ W_1(x), \ldots, W_q(x) ], ~
   D(x) = \diag(B_1(x), \ldots, B_q(x))
\end{equation}
is equivalent to $P(x)$, i.e., there exist unimodular
$q\times q$ matrix polynomials $E(x)$, $F(x)$ such that
$E(x)A(x)F(x)=I_{m(q-1)}\oplus P(x)$.
\end{theorem}

 \begin{proof} 
Define $E_0(x)$ the following (constant) matrix: 
   \[
     E_0(x) = \begin{bmatrix} 
       I_m & -I_m \\
         & I_m  & -I_m \\
         &    & \ddots  & \ddots \\
         &    &         & I_m & -I_m \\
         &    &         &   & I_m \\
	\end{bmatrix}.
   \]
A direct inspection shows that
   \[
     E_0(x) A(x) = \begin{bmatrix}
       B_1(x) & - B_2(x) \\
                    & B_2(x)   & -B_3(x) \\
                    &                & \ddots  & \ddots \\
                    &                &         & B_{q-1}(x) & -B_q(x) \\
       W_1(x)          &  W_2(x)          & \cdots  & W_{q-1}(x)     & B_q(x) + W_q(x) \\
	\end{bmatrix}. 
   \]
   Using the fact that the polynomials $B_i(x)$ are right co-prime, 
   we transform the latter matrix into block diagonal form. 
   We start by cleaning $B_1(x)$. Since $B_1(x), B_2(x)$ are right co-prime,
   there exist polynomials $\alpha(x)$, $\beta(x)$ such that
   $
     B_1(x) \alpha(x) + B_2(x) \beta(x) = I_m. 
   $
   For the sake of brevity, from now on we write $\alpha,\beta$, $W_i$ and $B_i$ in place of $\alpha(x),\beta(x)$, $W_i(x)$ and $B_i(x)$, respectively. 
Observe that the matrix
   \[
     F_1(x) = \begin{bmatrix} 
       \alpha & B_2 \\
       -\beta & B_1 \\
	\end{bmatrix} \oplus I_{m(q-2)}.
   \]
   is unimodular in view of Lemma~\ref{lem:unimodular_f1}, moreover
   \[
     E_0(x) A(x) F_1(x) = \begin{bmatrix}
       I_m & \\
       -B_2 \beta   & B_1B_2   & -B_3 \\
                    &                & \ddots  & \ddots \\
                    &                &         & B_{q-1} & -B_q \\
       W_1 \alpha - W_2 \beta  & W_1B_2 + W_2B_1  & \cdots  & W_{m-1}     & B_q + W_q \\
	\end{bmatrix}. 
   \]
   Using row operations we transform to zero all the elements in the first
   column of this matrix (by just adding multiples of the first row to
   the others). That is, there exists a suitable unimodular matrix $E_1(x)$ such
   that
   \[
	 E_1(x) E_0(x) A(x) F_1(x) = \begin{bmatrix}
       I_m & \\
                    & B_1B_2    & -B_3  \\
                    &                & \ddots  & \ddots \\
                    &                &         & B_{q-1} & -B_q \\
                    & W_2B_1 + W_1B_2  & \cdots  & W_{q-1}     & 
                      B_q  + W_q \\
	\end{bmatrix}.      
   \]
   In view of Lemma \ref{lem:primeproduct}, $B_1B_2$ is right coprime
   with $B_3$. Thus, we can recursively apply the same process until
   we arrive at the final reduction step: {
\small \[
     E_{q-1}(x) \ldots E_0(x) A(x) F_1(x)\ldots F_{q-1}(x) =
 I_{m(q-1)}\oplus \left( \prod_{i=1}^q B_i(x)+\sum_{i = 1}^q 
                                     W_i(x)C_i(x) \right)
\]
   }
   where the last diagonal block is exactly
   $P(x)$ in view of assumption \ref{cond:1}. 
\end{proof}

Observe that if $m=1$  and $B_i(x) = x
- \beta_i$ then we find that  \eqref{eq:secularlin}
provides the secular linearization for $P(x)$.

Now, we can prove that the polynomial equivalence that we have
 just presented is actually a strong equivalence, under the following additional assumptions
\begin{equation}\label{eq:addass}
\begin{split}
&\deg B_i(x)=d, \quad i=1,\ldots,q,\\
& n=dq,\quad \deg W_i(x)< \deg B_i(x)\\
& B_i^\#(x), \hbox{ are pairwise right co-prime}.
\end{split}
\end{equation}

\begin{remark}
  Recall that, according to Theorem 7.5 of \cite{ellification}
  in an algebraically closed field there exists a strong
  $\ell$-ification for a degree $n$ 
  regular matrix polynomial $P(x) \in \mathbb{C}^{m \times m}[x]$
  if and only if $\ell | nm$. Conditions~\eqref{eq:addass}
  satisfy this requirement, since $\ell = d$ and $d | n$. 
\end{remark}

 To accomplish this task we will show that the reversed matrix polynomial of
 $A(x)=D(x)+(e\otimes I_m)W$, $W=[W_1,\ldots,W_q]$,
 has the same structure as $A(x)$ itself.
  
 \begin{theorem} \label{thm:stronglin}
Under the assumptions of Theorem \ref{thm:3} and of \eqref{eq:addass}
   the secular $\ell$-ification given in Theorem~\ref{thm:3} is strong.
 \end{theorem}
 
 \begin{proof} 
 Consider  $A^\#(x)=x^dA(x^{-1})$. We have
 \[
   A^\#(x) = \diag(
     x^dB_1(x^{-1}),\ldots, x^dB_q(x^{-1})) + (e\otimes I_m) [ x^dW_1(x^{-1}) , \cdots, x^d W_q(x^{-1}) ].
 \]
 This matrix polynomial is already in the same form as $A(x)$ of \eqref{eq:A(x)} and verifies the assumptions of Theorem \ref{thm:3} since the polynomials
 $x^dB_1(x^{-1})$ are pairwise right co-prime and commute
 with each other in view of equation \eqref{eq:addass}.
Thus, Theorem \ref{thm:3} implies that $A^\#(x)$ is an $\ell$-ification for 
\[
\begin{split}
\prod_{i=1}^q x^dB_i(x^{-1}) &+\sum_{i=1}^q x^dW_i(x^{-1})\prod_{j\ne i}x^dB_j(x^{-1})\\
&=
x^{dq}\left(\prod_{i=1}^q B_i(x^{-1})+\sum_{i=1}^q W_i(x^{-1})\prod_{j\ne i}B_j(x^{-1})  \right)\\
&=x^nP(x^{-1})=P^\#(x),
\end{split}\]
where the first equality follows from the fact that $n=dq$ in view of equation \eqref{eq:addass}. This concludes the proof.
\end{proof}
 
 \begin{remark}
   Note that in the case where the $B_i(x)$ do not have the same
   degree, the secular $\ell$-ification might not be strong: the finite
   eigenstructure is preserved but some infinite eigenvalues not
   present in $P(x)$ will be artificially introduced.
 \end{remark}

 \section{Eigenvectors}\label{sec:eig}
In this section we provide an explicit expression of right and left
eigenvectors of the matrix polynomial $A(x)$.

 \begin{theorem} \label{thm:righteigenvec}
   Let $P(x)$ be a matrix polynomial, $A(x)$ its secular
   $\ell$-ification defined in Theorem \ref{thm:3}, $\lambda\in\mathbb
   C$ such that $\det P(\lambda)=0$, and assume that $\det
   B_i(\lambda)\ne 0$ for all $i = 1, \ldots, q$.  
 If $v_A=(v_1^t,\ldots,v_q^t)^t\in\mathbb C^{mq}$ is such that
 $A(\lambda)v_A=0$, $v_A\ne 0$ then $P(\lambda)v=0$ where
 $v=-\prod_{i=1}^qB_i(\lambda)^{-1}\sum_{j=1}^q W_jv_j\ne 0$. Conversely, if
$v\in\mathbb C^m$ is a nonzero vector such that $P(\lambda)v=0$, then the vector
$v_A$ defined by $v_i=\prod_{j \neq i} B_j(\lambda)v$, $i=1,\ldots,q$ is nonzero and such that
$A(\lambda)v_A=0$.
 \end{theorem}
 
\begin{proof}
Let $v_A\ne 0$ be such that $A(\lambda)v_A=0$, so that 
\begin{equation}\label{eq:th1}
B_i(\lambda)v_i+\sum_{j=1}^q W_j(\lambda)v_j=0,\quad i=1,\ldots,q.
\end{equation}
Let $v=-(\prod_{i=1}^q B_i(\lambda)^{-1})\sum_{j=1}^qW_j(\lambda)v_j$. Combining the latter equation and \eqref{eq:th1} yields
\begin{equation}\label{eq:thm3.1}
v_i=-B_i(\lambda)^{-1}\left(\sum_{j=1}^q W_j(\lambda)v_j\right)=\prod_{j=1,\,j\ne i}^q B_j(\lambda)v.
\end{equation}
Observe that if $v=0$ then, by definition of $v$, one has $\sum_{j=1}^q W_j(\lambda)v_j=0$ so that,
in view of \eqref{eq:th1}, we find that $B_i(\lambda)v_i=0$. Since $\det B_i(\lambda)\ne0$ this would imply that $v_i=0$ for any $i$ so that $v_A=0$ which contradicts the assumptions.
Now we prove that $P(\lambda)v=0$. In view of \eqref{eq:thm3.1} we have 
\[
P(\lambda)v=\prod_{j=1}^qB_j(\lambda)v+\sum_{i=1}^qW_i(\lambda)
\prod_{j=1,\,j\ne i}^qB_j(\lambda)v= \prod_{j=1}^qB_j(\lambda)v+\sum_{i=1}^qW_i(\lambda)v_i.
\]
Moreover, by definition of $v$ we get
\[
P(\lambda)v= -\prod_{j=1}^qB_j(\lambda)(\prod_{i=1}^q B_i(\lambda)^{-1})\sum_{i=1}^qW_i(\lambda) v_i+\sum_{i=1}^qW_i(\lambda)v_i=0.
\]
Similarly, we can prove the opposite implication.
 \end{proof}
 
 A similar result can be proven for left eigenvectors. The following theorem relates left eigenvectors of $A(x)$ and left eigenvectors of $P(x)$.
 
 \begin{theorem} Let $P(x)$ be a matrix polynomial, $A(x)$ its secular
   $\ell$-ification defined in Theorem \ref{thm:3}, $\lambda\in\mathbb
   C$ such that $\det P(\lambda)=0$, and assume that $\det
   B_i(\lambda)\ne 0$.  If $u_A^t=(u_1^t,\ldots,u_q^t)\in\mathbb C^{mq}$ is such that
   $u_A^t A(\lambda)=0$, $u_A\ne 0$, then $u^tP(\lambda)=0$ where
   $u=\sum_{i=1}^q u_i\ne 0$. Conversely, if $u^tP(\lambda)=0$ for a nonzero vector $u\in\mathbb C^m$ then
   $u_A^t A(\lambda)=0$, where $u_A$ is a nonzero vector defined by $u_i^t=-u^tW_i(\lambda)B_i(\lambda)^{-1}$ for
   $i=1,\ldots, q$.
 \end{theorem}
\begin{proof}
If $u_A^t A(\lambda)=0$ then from the expression of $ A(x)$ given in
Theorem~\ref{thm:3} we have
\begin{equation}\label{eq:thm3}
u_i^tB_i(\lambda)+\left(\sum_{j=1}^q u_j^t \right)W_i(\lambda)=0,\quad i=1,\ldots, q.
\end{equation}
Assume that $u=\sum_{j=1}^q u_j=0$. Then from the above expression we obtain, for any $i$, $u_i^tB_i(\lambda)=0$ that is $u_i=0$ for any $i$ since 
$\det B_i(\lambda)\ne0$. This is in contradiction with $u_A\ne 0$. From \eqref{eq:thm3} we obtain
$u_i^t=-u^tW_i(\lambda)B_i(\lambda)^{-1}$. Moreover,
multiplying \eqref{eq:thm3} to the right by $\prod_{j=1,\,j\ne i}^q B_j$ yields
\[
0=u_i^t\prod_{j=1}^q B_j(\lambda)+u^tW_i(\lambda)\prod_{j=1,\,j\ne i}^qB_j(\lambda).
\]
Taking the sum of the above expression for $i=1,\ldots,q$ yields
\[
0=\left(\sum_{i=1}^q u_i^t\right)\prod_{j=1}^q B_j(\lambda)+u^t\sum_{i=1}^q W_i(\lambda)\prod_{j=1,\,j\ne i}^qB_j(\lambda)=u^tP(\lambda).
\]
Conversely, assuming that $u^tP(\lambda)=0$, from the representation 
\[
P(x)=\prod_{j=1}^n B_j(x)+\sum_{i=1}^q W_i(x)\prod_{j=1,\,j\ne i}^q B_i(x),
\]
 defining $u_i^t=-u^tW_i(\lambda)B_i(\lambda)^{-1}$ we obtain 
\[
\sum_{i=1}^qu_i^t=-u^t\sum_{i=1}^q W_i(\lambda)B_i(\lambda)^{-1}=-u^t(P(\lambda)\prod_{j=1}^qB_j(\lambda)^{-1}-I)=u^t
\] and therefore from \eqref{eq:thm3} we deduce that $u_A^tA(\lambda)=0$.
 \end{proof}

The above result does not cover the case where $\det B_i(\lambda)=0$ for some $i$.

\subsection{A sparse $\ell$-ification} 
Consider the block bidiagonal matrix $L$ having $I_m$ on the block
diagonal and $-I_m$ on the block subdiagonal. It is immediate to
verify that $L(e\otimes I_m)=e_1\otimes I_m$, where
$e_1=(1,0,\ldots,0)^t$. This way, the matrix polynomial $H(x)=LA(x)$ 
is a sparse $\ell$-ification of the form
\[
H(x)=\left[\begin{array}{ccccc}
B_1(x)+W_1(x)&W_2(x)&\ldots&W_{q-1}(x)&W_q(x)\\
-B_1(x)&B_2(x)\\
      &-B_2(x)&\ddots\\
      &      &\ddots&B_{q-1}(x)\\
      &      &      &-B_{q-1}(x)&B_q(x)
\end{array}\right]
\]

\section{A particular case}\label{sec:particular}
In the previous section we have provided (strong) $\ell$-ifications of
a matrix polynomial $P(x)$ under the assumption of the existence of
the representation
 \begin{equation}\label{eq:repr}
P(x)=\prod_{i=1}^q B_i(x)+\sum_{i=1}^q W_i(x)\prod_{j\ne i}B_j(x)
\end{equation}
and under suitable conditions on $B_i(x)$. In this section we show
that a specific choice of the blocks $B_i(x)$ satisfies the above assumptions
and implies the existence of the representation \eqref{eq:repr}. Moreover, we
provide explicit formulas for the computation of $W_i(x)$ given $P(x)$
and $B_i(x)$.

We provide also some  additional conditions in order to make the
resulting $\ell$-ification strong. 

\begin{assumption}\label{ass:1}
The matrix polynomials $B_i(x)$ are defined as follows
\[
 \begin{cases}
B_i(x)=b_i(x)I & i=1,\ldots,q-1\\
B_q(x)=b_q(x)P_n+s I & \text{otherwise} \\
\end{cases}\]
where $b_i(x)$ are scalar polynomials such that $\deg b_i(x)=d_i$ and $\sum_{i=1}^q d_i=n$; the polynomials $b_i(x)$, $i=1,\ldots,q$,  are pairwise co-prime; 
 $s$ is a
  constant such that  $\lambda b_q(\xi) + s \neq 0$  for any eigenvalue $\lambda$ of $P_n$ and for any root $\xi$ of $b_i(x)$, for $i = 1, \dots, q-1$.
\end{assumption}

In this case it is possible to prove the existence of the representation
\eqref{eq:repr}.
We rely on the Chinese remainder theorem that here we rephrase in terms of matrix polynomials.

\begin{lemma}\label{lem:chin}
Let $b_i(x)$, $i=1,\ldots,q$ be co-prime polynomials of degree
$d_1,\ldots,d_q$, respectively, such that $\sum_{i=1}^q d_i=n$. If
$P_1(x)$, $P_2(x)$ are matrix polynomials of degree at most $n-1$ then
$P_1(x)=P_2(x)$ if and only if $P_1(x)-P_2(x)\equiv 0\mod b_i(x)$, for
$i=1,\ldots,q$.
\end{lemma}
\begin{proof}
The implication $P_1(x)-P_2(x)=0$ $\Rightarrow$ $P_1(x)-P_2(x)\equiv
0\mod b_i(x)$ is trivial. Conversely, if $P_1(x) - P_2(x) \equiv 0
\mod b_i(x)$ for every $b_i$ then the entries of $P_1(x)-P_2(x)$ are
multiples of $\prod_{i = 1}^q b_i(x)$ for the co-primality of the polynomials $b_i(x)$. But this implies that $P_1(x) -
P_2(x) = 0$ since the degree of $P_1(x)-P_2(x)$ is at most $n-1$ while
$\prod_{i = 1}^q b_i(x)$ has degree $n$.
\end{proof}

We have the following

\begin{theorem} \label{thm:1}
  Let $P(x) = \sum_{i = 0}^n x^i P_i$ be an $m \times m$ matrix
  polynomial over an algebraically closed field. Under Assumption \ref{ass:1}, set $C_i(x) = \prod_{j \neq i}
  B_j(x)$. Then there exists a
  unique decomposition
  \begin{equation}\label{eq:thm}
    P(x) = B(x) + \sum_{i = 1}^q W_i(x) C_i(x), \qquad
    B(x) = \prod_{i = 1}^q B_i(x),
  \end{equation}
  where $W_i(x)$ are matrix polynomials of degree less than $d_i$ for
  $i = 1, \dots, q$ defined by
\begin{equation}\label{eq:ofb}
\begin{split}
&W_i(x)=\frac{P(x)}{\prod_{j=1,\, j\ne i}^{q-1}b_j(x)}(b_q(x)P_n+sI_m)^{-1}\mod b_i(x),\quad i=1,\ldots,q-1\\
&W_q(x)=\frac 1{\prod_{j=1}^{q-1}b_j(x)}P(x)-sI_m-s\sum_{j=1}^{q-1}\frac {W_j(x)}{b_j(x)}\mod b_q(x).
\end{split}
\end{equation}
\end{theorem}

\begin{proof}
We show that there exist matrix polynomials $W_i(x)$ of degree less
than $d_i$ such that $P(x)-B(x)\equiv \sum_{i = 1}^q W_i(x) C_i(x)\mod
b_i(x)$ for $i=1,\ldots, q$. Then we apply Lemma \ref{lem:chin} with
$P_1(x)=P(x) - B(x)$ that by construction has degree at most $n-1$,
and with $P_2(x) = \sum_{i = 1}^q W_i(x) C_i(x)$, and conclude that
$P(x)=B(x)+ \sum_{i = 1}^q W_i(x) C_i(x)$.  Since for $i=1,\ldots,q-1$
the polynomial $b_i(x)$ divides every entry of $B(x)$ and of $C_j(x)$
for $j \neq i$, we find that $ P(x) \equiv W_i(x)C_i(x) \mod
b_i(x),\quad i=1,\ldots,q.  $ Moreover, for $i<q$ we have $C_i(x) =
\left( \prod_{j \neq i, j < q} b_j(x) I_m \right) (b_q(x) P_n + s I_m
)$.  The first term is invertible modulo $b_i(x)$ since by assumption
$b_i(x)$ is co-prime with $b_j$ for every $j \neq i$.  We need to
prove that the matrix on the right is invertible modulo $b_i(x)$, that
is, its eigenvalues $\mu$ are such that $b_i(\mu) \neq 0$. Now, since
the eigenvalues of $b_q(x) P_n + s I_m $ have the form
$\mu=b_q(x)\lambda+s$, where $\lambda $ is an eigenvalue of
$P_n$, it is enough to ensure that for every $\xi$ which is a root of $b_i(x)$
the value $\lambda b_q(\xi) + s$ is different from $0$ for
$i=1,\ldots,q-1$.  This is guaranteed by hypothesis, and so we obtain
the explicit formula for $W_i(x)$, $i=1,\ldots,q-1$ given by
\eqref{eq:ofb}.  It remains to find an explicit expression for
$W_q(x)$. We have $W_q(x)C_q(x)=P(x)-\sum_{j=1}^{q-1}W_j(x)C_j(x)$,
where the right-hand side is made by known polynomials. This way,
taking the latter expression modulo $b_q(x)$ we can compute $W_q(x)$
since $C_q(x)=\prod_{j=1}^{q-1}b_j(x)I_m$ is invertible modulo
$b_q(x)$ in view of the co-primality of the polynomials
$b_1(x),\ldots,b_q(x)$. This way we get the expression of $W_q$ in
\eqref{eq:ofb}.
\end{proof}

\begin{remark}
  Note that in the case where $P_n = I$ it is possible 
  to choose $s = 0$ 
  so that Equations~\eqref{eq:ofb} take a simpler form.
\end{remark}

We observe that the matrix polynomials $B_i(x)$ which satisfy
Assumption \ref{ass:1} verify the hypotheses of Theorem
\ref{thm:3}. Therefore there exists an $\ell$-ification of $P(x)$
which can be computed. In view of Theorem~\ref{thm:stronglin}
we have that this $\ell$-ification is also strong if the following conditions
are satisfied: 
\begin{enumerate}
  \item The $B_i(x)$ have the same degree $d$. In our case this implies
    that $\deg b_i(x) = d$ for every $i = 1, \dots, q$. 
  \item The matrix polynomials $x^d B_i(x^{-1})$ are right coprime. It can be
    seen that under Assumption~\ref{ass:1} this is equivalent to 
    asking that $b_i(0) \ne 0$ for every $i = 1, \dots, q$ and that
    either $P_n\ne I$ or $\xi^d b_q(\xi) + s \ne 0$ for every $\xi$ root of
    $b_i(x)$, for $i < q$. 
\end{enumerate}

Here we provide an example of an $\ell$-ification of degree $2$ for a
$2 \times 2$ matrix polynomial $P(x)$ of degree $4$. 

\begin{example}\rm Let
\[
 P(x) = \begin{bmatrix}
  x^4+2 & -1 \\
  x & x^3-1 \\
 \end{bmatrix}, \quad
 b_1(x) = x^2-2, \quad
 b_2(x) = x^2 + 2, \quad
 s = 1.
\]
Applying the above formulas we obtain
\[
  W_1(x) = \begin{bmatrix}
    \frac{6}{5} & -1 \\
    \frac 15 x  & -1+2x \\
  \end{bmatrix}, \qquad
  W_2(x) = \begin{bmatrix}
   -\frac{11}{5} & 0 \\
   -\frac 15 x   & -1+x \\
  \end{bmatrix}. 
\]
Then we have that $A(x)$ is a degree $2$ $\ell$-ification for $P(x)$,
that is, a quadratization, by setting
\[
  A(x) = \begin{bmatrix}
    x^2 - 2 & 0 & 0 & 0\\
    0 & x^2 - 2 & 0 & 0 \\
    0 & 0 & x^2 + 3 & 0 \\
    0 & 0 & 0 & 1 \\
  \end{bmatrix} + \begin{bmatrix}
    \frac{6}{5} & -1 & -\frac{11}{5} & 0 \\
    \frac 15 x  & -1+2x & -\frac 15x & -1+x \\
   \frac{6}{5} & -1 & -\frac{11}{5} & 0 \\
   \frac 15x  & -1+2x & -\frac 15 x & -1+x \\
  \end{bmatrix}.
\]
\end{example}

In the case where $b_i(x)$, $i=1,\ldots,q$ are linear polynomials we
have the following:

\begin{corollary}\label{cor:d=1} If  $b_i(x)=x-\beta_i$, $i=1,\ldots,q$, then $q = n$ and
\[
\begin{split}
&W_i=\frac{P(\beta_i)} {\prod_{j=1,\, j\neq i} ^ {n-1} (\beta_i - \beta_j )}((\beta_i-\beta_n) P_n +s I_m)^{-1}, \quad i=1,\ldots,n-1,\\[1ex]
&W_n= \frac{ P(\beta_n)}{\prod_{j=1}^{n-1}(\beta_n-\beta_j)} - sI_m-s\sum_{j = 1}^{n-1} \frac{W_j}{\beta_n - \beta_j}.
\end{split}
\]
Moreover, if $P(x)$ is monic
then with $s=0$ the expression for $W_i$ turns simply into
$W_i=P(\beta_i)/\prod_{j=1,\, j\ne i}^n(\beta_i-\beta_j)$, for $i=1,\ldots,n$.
\end{corollary}
\begin{proof}
It follows from Theorem \ref{thm:1} and from the property $v(x)\mod x-\beta=v(\beta)$ valid for any polynomial $v(x)$.
\end{proof}

Given $n$ and $q\le n$, let $\ell=\lceil\frac nq\rceil$.  We may choose
polynomials $b_i(x)$ of degree $d_i$ in between $\ell-1$ and $\ell$ such
that $\sum_{i=1}^q d_i=n$. This way  $A(x)$ is an $mq\times mq$ matrix polynomial of degree $\ell$. For
instance, if $\ell=2$ we obtain a quadratization of $P(x)$.

If $P(x)$ is monic, that is $P_n=I$, we can handle another particular case of $\ell$-ification by
choosing
diagonal matrix polynomials $B_i(x)$. 

Let
$B_i(x)=\diag(d^{(i)}_1(x),\ldots,d^{(i)}_m(x))=:D_i(x)$ be monic
matrix polynomials such that the corresponding diagonal entries of $D_i(x)$
and $D_j(x)$ are pairwise co-prime for any $i\ne j$ so that the
second assumption
of Theorem \ref{thm:3} is satisfied. Let us prove that there exist
matrix polynomials $W_i(x)$ such that $\deg W_i(x)<\deg D_i(x)$ and
\begin{equation}\label{eq:dd}
P(x)=\prod_{i=1}^q D_i(x)+\sum_{i=1}^q W_i(x)C_i(x),\quad 
C_i(x)=\prod_{j=1,\,j\ne i}^qD_i(x),
\end{equation}
so that Theorem \ref{thm:3} can be applied.
Observe that
equating the coefficients of $x^i$ in \eqref{eq:dd} for $i=0,\ldots,n-1$
provides a linear system of 
$m^2n$ equations in $m^2n$ unknowns.
Equating the $j$th columns of both sides of \eqref{eq:dd} modulo  
$d^{(i)}_j(x)$ yields
\[
P(x)e_j\mod d^{(i)}_j(x) =\prod_{s=1,\, s\ne i}^n d_j^{(s)} W_i(x)e_j \mod d^{(i)}_j(x),\quad 
i=1,\ldots,m.
\]
 The
above equation allows one to compute the coefficients of the
polynomials of degree at most $\deg D_j(x)-1$ in the $j$th column of
$W_j(x)$ by means of the Chinese remainder theorem.

\section{Computational issues}\label{sec:comput}
In the previous section we have given explicit formulas for the
(strong) $\ell$-ification of a matrix polynomial satisfying Assumption
\ref{ass:1}. Here we describe some algorithms for the computation of
the matrix coefficients $W_i(x)$.

In the case where $d=1$, the equations given in Corollary
\ref{cor:d=1} provide a straightforward algorithm for the computation
of the matrices $W_i$ for $i=1,\ldots, n$. In the case where $d_i=\deg
b_i(x)>1$ for some values of $i$ we have to apply \eqref{eq:ofb} which
involve operations modulo scalar polynomials $b_i(x)$ for
$i=1,\ldots,q$.

The main computational issues in this case are the evaluation of a
scalar polynomial modulo a given $b_i(x)$, the evaluation of the
inverse of a scalar polynomial modulo $b_i(x)$ and the more
complicated task of evaluating the inverse of a matrix polynomial
modulo $b_i(x)$. 

In general we recall that, given polynomials $v(x)$
and $b(x)$ such that $v(x)$ is co-prime with $b(x)$, there exist
polynomials $\alpha(x)$ and $\beta(x)$ such that
\begin{equation}\label{eq:alfa}
\alpha(x)v(x)+\beta(x)b(x)=1.
\end{equation}
This way, we have $\alpha(x)=1/v(x)\mod
b(x)$. 

There are algorithms for computing the coefficients of
$\alpha(x)$ given the coefficients of $v(x)$ and $b(x)$. We refer the
reader to the book \cite{bp:book} and to any textbook in computer
algebra for the design and analysis of algorithms for this
computation.  Here, we recall a simple numerical technique, based on the
evaluation-interpolation strategy, which can be also directly applied to
the matrix case.

Observe that from \eqref{eq:alfa} it turns out that
$\alpha(\xi_j)=1/v(\xi_j)$ for $j=1,\ldots,d$, where $\xi_j$ are the
zeros of the polynomial $b(x)$ of degree $d$. Since $\alpha(x)$ has
degree at most $d-1$, it is enough to compute the values $1/v(\xi_j)$
in order to recover the coefficients of $\alpha(x)$ through an
interpolation process. This procedure amounts to $d$ evaluations of a
polynomial at a point and to solving an interpolation problem for the
overall cost of $O(d^2)$ arithmetic operations. If the polynomial
$b(x)$ is chosen in such a way that its roots are multiples of the
$d$th roots of unity then the evaluation/interpolation problem is
well conditioned, and it can be performed by means of FFT which is a fast
and numerically stable procedure.

The evaluation/interpolation technique can be extended to the case of
matrix polynomials. For instance, the computation of the coefficients
of the matrix polynomial $F(x)=V(x)^{-1}\mod b(x)$, where $V(x)$ is a
given matrix polynomial co-prime with $b(x)I$, can be performed in the
following way:
\begin{enumerate}
\item compute $Y_k=V(\xi_k)^{-1}$, for $k=1,\ldots,d$;
\item for any pair $(i,j)$, interpolate the entries $y^{(k)}_{i,j}$, $k=1,\ldots,d$ of the matrix $Y_k$ and find the coefficients of the polynomial $f_{i,j}(x)$, where $F(x)=(f_{i,j}(x))$.
\end{enumerate}
This procedure requires the evaluation of $m^2$ polynomials at $d$
points, followed by the inversion of $d$ matrices of order $m$ and the
solution of $m^2$ interpolation problems. The cost turns to
$O(m^2d^2)$ ops for the evaluation stage, $O(m^3d)$ ops for the
inversion stage, and $O(m^2d^2)$ for the interpolation stage. In the
case where the polynomial $b(x)$ is such that its roots are multiple
of the $d$ roots of the unity, the evaluation and the interpolation
stage have cost $O(m^2d\log d)$ if performed by means of FFT.

Observe that, in the case of polynomials $b_i(x)$ of degree one, the above procedure coincides with the one provided directly by equations in Corollary
 \eqref{cor:d=1}.

 \section{Numerical issues}\label{sec:app}
Let $\omega_n$ be a principal $n$th root of the unity, define
$\Omega_n=\frac1{\sqrt n}(\omega_n^{ij})_{i,j=1,n}$ the Fourier matrix
such that $\Omega_n^*\Omega_n=I_n$ and observe that $\Omega_n e=e_n$
where $e=(1,\ldots,1)^t$, $e_n=(0,\ldots,0,1)^t$.  Assume for simplicity
$P_n=I_m$. For the
linearization obtained with $\beta_i=\omega_n^i$, $i=1,\ldots,n$, we
have, following Corollary~\ref{cor:d=1}, 
\[
A(x)=x I_{mn}-\diag(\omega_n^1 I_m,\omega_n^2
I_m,\ldots,\omega_n^n I_m)+(e\otimes I_m)[W_1,\ldots,W_n]
\]
with $W_i=\frac 1n \omega_n^iP(\omega_n^i)$. The latter equation
follows from $W_i = P(\omega_n^i) / (\prod_{j \neq i} (\omega_n^i - \omega_n^j))$ since $\prod_{j \neq i} (\omega_n^i - \omega_n^j)$ coincides with the first 
derivative of $x^n - 1 = \prod_{j = 1}^n (x - \omega_n^j)$ evaluated at $x = \omega_n^i$, that is $\prod_{j \neq i} (\omega_n^i - \omega_n^j) = n \omega_n^{-i}$. 
It is easy to verify that the pencil $(\Omega_n^*\otimes I_m)A(x)(\Omega_n
\otimes I_m)$ has the form
\[
x I_{mn}-F,\quad F=(C\otimes I_m)-\begin{bmatrix} P_0+I_m \\P_1 \\ \vdots \\P_{n-1}\end{bmatrix} (e_n^t \otimes I_m)
\]
where $C=(c_{i,j})$ is the unit circulant matrix defined by
$c_{i,j}=(\delta_{i,j+1\mod n})$. That is, $F$ is the block Frobenius
matrix associated with the matrix polynomial $P(x)$.

This shows that our linearization includes the companion Frobenius
matrix with a specific choice of the nodes. In particular, since
$\Omega_n$ is unitary, the condition number of the eigenvalues of
$A(x)$ coincides with the condition number of the eigenvalues of $F$.
Observe also that if we choose $\beta_i=\alpha\omega_n^i$ with
$\alpha\ne 0$, then $(\Omega_n^* \otimes I_m) A(x)(\Omega_n \otimes I_m)=xI-D_\alpha^{-1}
FD_\alpha$ for $D_\alpha=\diag(1,\alpha,\ldots,\alpha^{n-1})$.
That is, we obtain a scaled Frobenius pencil.

Here, we present some numerical experiments to show that in many
interesting cases a careful choice of the $B_i(x)$ can lead to
linearizations (or $\ell$-ifications) where the eigenvalues are much
better conditioned than in the original problem.  Here we are
interested in measuring the conditioning of the eigenvalues of a
pencil built using these different strategies.  Recall that the
conditioning of an eigenvalue $\lambda$ of a matrix pencil $xA - B$
can be bounded by $\kappa_\lambda \leq \frac{\lVert v \rVert \lVert w
  \rVert}{|w^* A v|}$ where $v$ and $w$ are the right and left
eigenvectors relative to $\lambda$, respectively \cite{tisseur}.  This
is the quantity measured by the {\tt condeig} function in MATLAB that
we have used in the experiments.  The above bound can be extended to a
matrix polynomial $P(x) = \sum_{i = 0}^n P_i x^i$. In particular, the
conditioning number of an eigenvalue $\lambda$ of $P$ can be bounded
by $\kappa_\lambda \leq \frac{\lVert v \rVert \lVert w \rVert}{|w^*
  P'(\lambda) v|}$ where $v$ and $w$ are the right and left
eigenvectors relative to $\lambda$, respectively, \cite{tisseur}.
Observe that, if $xA - B$ is the Frobenius linearization of a matrix
polynomial $P(x)$, then the condition number of the eigenvalues of the
linearization is larger than the one concerning $P(x)$, since the
perturbation to the input data on $xA - B$ can be seen as a larger set
with respect to the perturbation to the coefficients of $P(x)$. This
is not true, in general, for a linearization in a different basis, as
in our case, since there is not a direct correspondence between the
perturbations on the original coefficients and the perturbations on
the linearization.  An analysis of the condition number for eigenvalue
problems of matrix polynomials represented in different basis is given
in \cite{corless}.

The code used to generate these examples can be downloaded from
\url{http://numpi.dm.unipi.it/software/secular-linearization/}. 

\subsection{Scalar polynomials}
As a first example, consider a monic scalar polynomial $p(x) = \sum_{i =
  0}^n p_i x^i$ where the coefficients $p_i$ have unbalanced moduli.
In this case, we generate $p_i$ using the MATLAB command 
{\tt \lstinline-p = exp(12 * randn(1,n+1));- 
\lstinline-p(n+1)=1;-  }

\begin{figure}[ht]
\begin{center}
  \includegraphics[width=7cm]{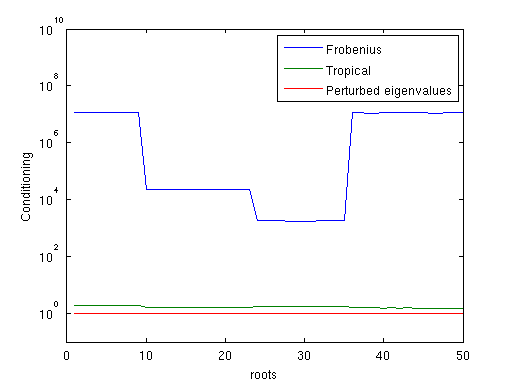}
  \caption{Conditioning of the eigenvalues of different linearizations of a degree $50$
    scalar polynomial with random unbalanced coefficients.}
     \label{fig:p50_test}
\end{center}
\end{figure}

Then we build our linearization by means of the function
{\tt \lstinline-seccomp(b,p)- }
that takes a vector {\tt \lstinline-b- } together with the
coefficients of the polynomial and generates the linearization $A(x)$
where $B_i(x) = x - \beta_i$ for $\beta_i={\tt b(i)}$.  Finally, we measure the conditioning of
the eigenvalues of $A(x)$ by means of the Matlab function
{\tt \lstinline-condeig-}.

We have considered three different linearizations:
\begin{itemize}
  \item The Frobenius linearization obtained by {\tt \lstinline-compan(p)-};
  \item the secular linearization obtained by taking  as $\beta_i$ some perturbed
    values of the roots; these values have been obtained by
    multiplying the roots by $(1 + \epsilon)$ with
    $\epsilon$ chosen randomly with
    Gaussian distribution $\epsilon \sim 10^{-12} \cdot N(0, 1)$. 
   \item the secular linearization with nodes given by the tropical roots of
     the polynomial multiplied by unit complex numbers.
\end{itemize}

The results are displayed in Figure~\ref{fig:p50_test}. One can see
that in the first case the condition numbers of the eigenvalues are
much different from each other and can be as large as $10^{8}$ for
the worst conditioned eigenvalue.  In the second case the condition
number of all the eigenvalues is close to $1$, while in the
third linearization the condition numbers are much smaller than those
of the Frobenius linearization and have an almost uniform distribution.

These experimental results are a direct verification of a conditioning
result of \cite[Sect. 5.2]{secsolve} that is at the basis of the {\tt secsolve}
algorithm presented in that paper.  
These tests are implemented in the function files {\tt \lstinline-Example1.m-} 
and {\tt \lstinline-Experiment1.m-} included
in the MATLAB source code for the experiments. 
A similar behavior of the conditioning for the eigenvalue 
problem holds in the matrix case.

\subsection{The matrix case} 
Consider now a matrix polynomial $ P(x) = \sum_{i = 0}^n P_i x^i $.
As in the previous case, we start by considering monic matrix
polynomials.  As a first example, consider the case where the
coefficients $P_i$ have unbalanced norms.
Here is the Matlab code that we have used to generate this test: 

{\tt
\begin{lstlisting}
n = 5; m = 64;
P = {};
for i = 1 : n
  P{i} = exp(12 * randn) * randn(m);
end
P{n+1} = eye(m);
\end{lstlisting}
}

We can give reasonable estimates to the modulus of the eigenvalues
using the Pellet theorem or the tropical roots. See \cite{gs09,
  pellet}, for some insight on these tools.

The same examples given in the scalar case have been replicated for matrix polynomials relying on the Matlab script published on
the website reported above by issuing the following commands: 

{\tt
\begin{lstlisting}
>> P = Example2();
>> Experiment2(P);
\end{lstlisting}
}

\begin{figure}[ht] 
  \centering
  \begin{minipage}{0.48\textwidth}
    \includegraphics[width=\textwidth]{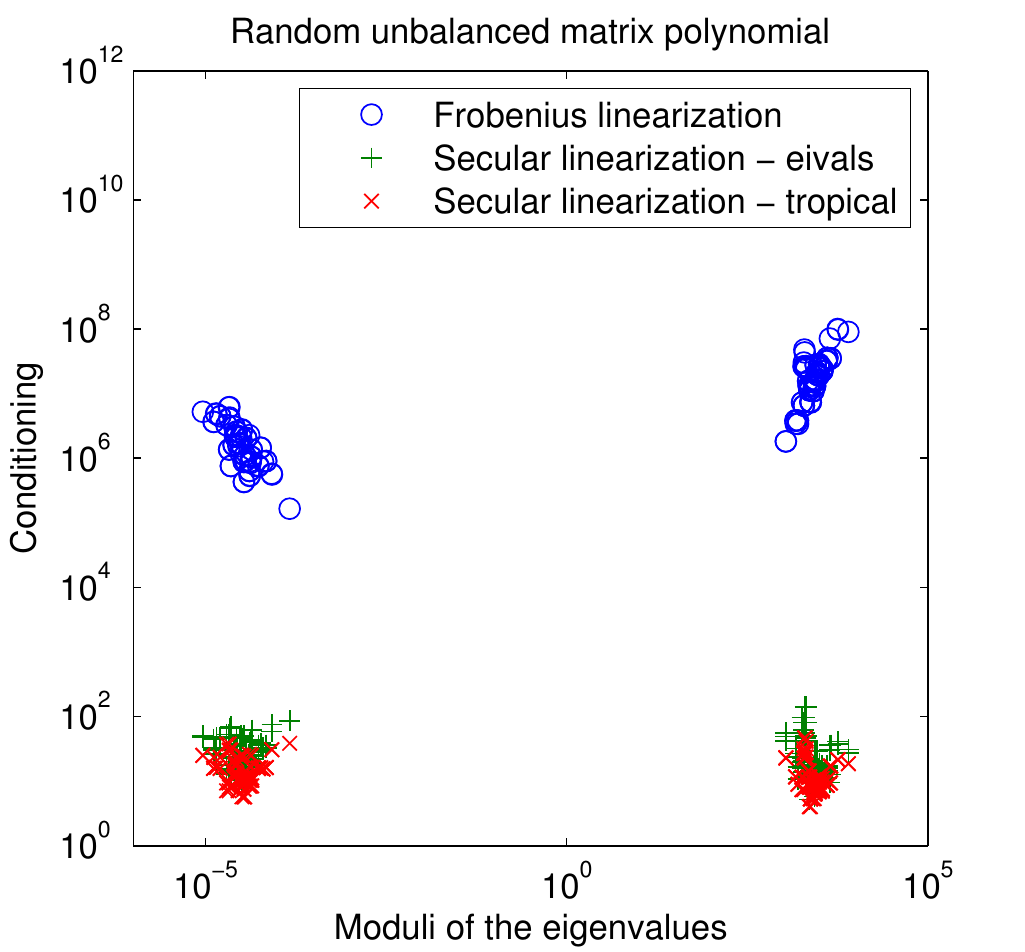}
  \end{minipage}
  ~
  \begin{minipage}{0.48\textwidth}
    \includegraphics[width=\textwidth]{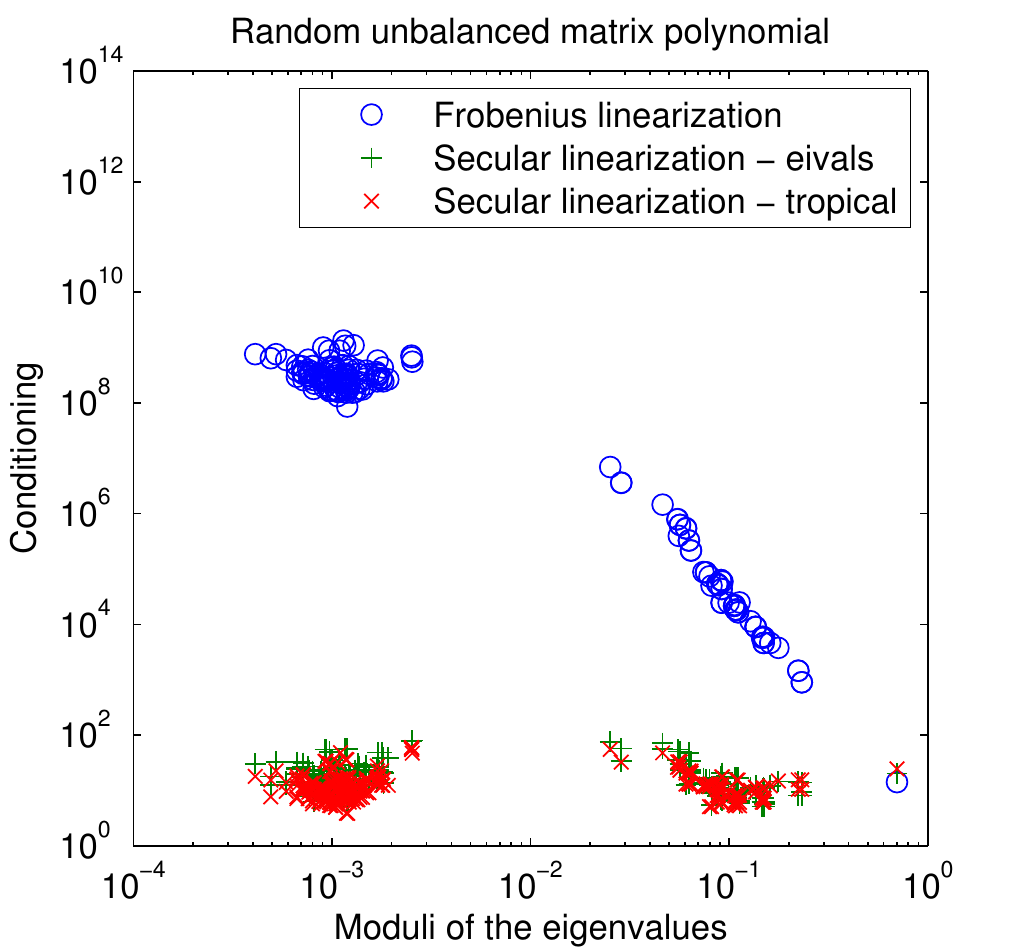}
  \end{minipage} \\[12pt]
  
  \begin{minipage}{0.48\textwidth}
    \includegraphics[width=\textwidth]{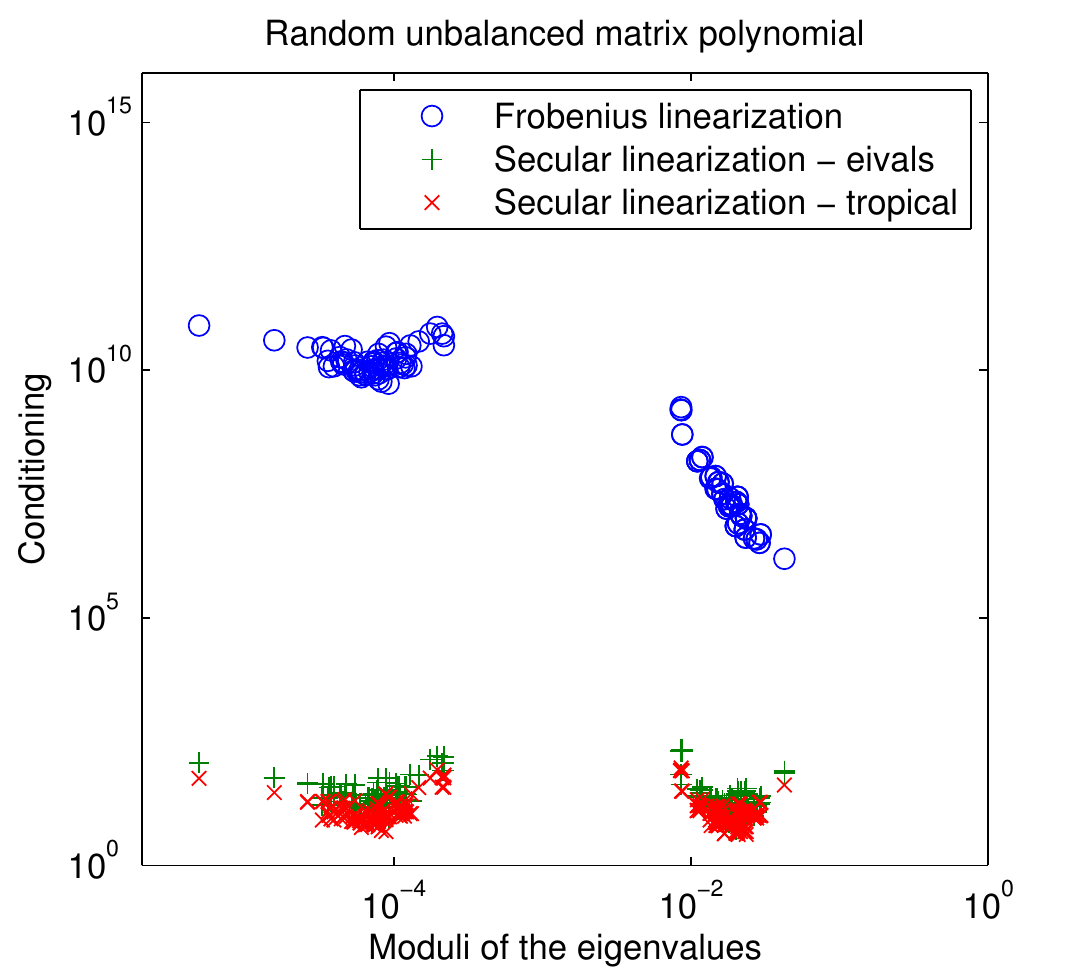}
  \end{minipage}
  ~
  \begin{minipage}{0.48\textwidth}
    \includegraphics[width=\textwidth]{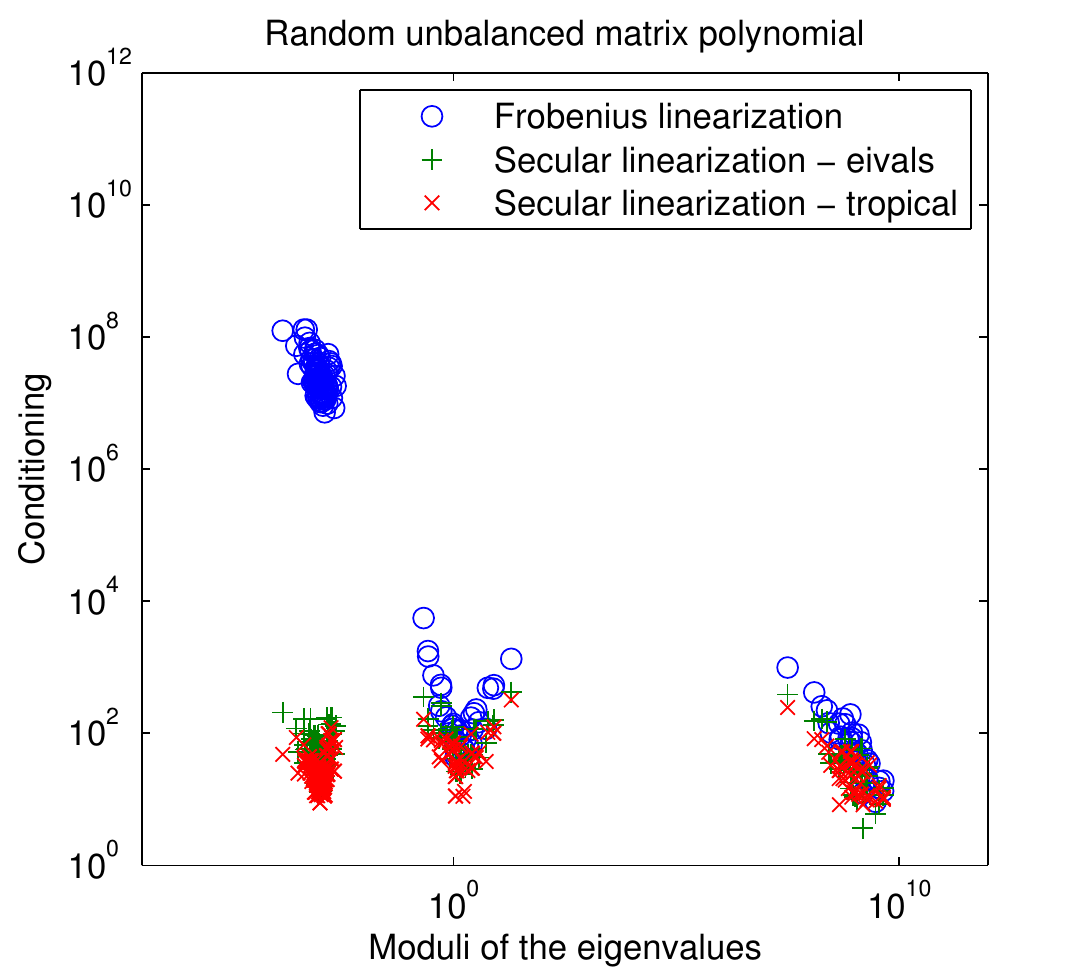}
  \end{minipage} 
  
  \caption{Conditioning of the eigenvalues of different linearizations for some matrix polynomials
with random coefficients having unbalanced norms. }
  \label{fig:randommatrixtest}
\end{figure}

We have considered three linearizations: the standard Frobenius
companion linearization, and two versions of our secular linearizations.  In
the first version the nodes $\beta_i$ are the mean of the moduli of
set of eigenvalues with close moduli multiplied by unitary complex
numbers. In the second, the values of $\beta_i$ are obtained by the
Pellet estimates delivered by the tropical roots.

In Figure~\ref{fig:randommatrixtest} we report the conditioning of the
eigenvalues, measured with Matlab's {\tt condeig}.

It is interesting to note that the conditioning of the secular
linearization is, in every case, not exceeding $10^3$. Moreover it can
be observed that no improvement is obtained on the conditioning of the
eigenvalues that are already well-conditioned. In contrast, there is a
clear improvement on the ill-conditioned ones. In this particular
case, this class of linearizations seems to give an almost uniform
bound to the condition number of all the eigenvalues.

Further examples come from the NLEVP collection of \cite{nlevp}. We
have selected some problems that exhibit bad conditioning.

As a first example we consider the problem {\tt
  orr\_sommerfeld}. Using the tropical roots we can find some values
inside the unique annulus that is identified by the Pellet theorem. In
this example the values obtained only give a partial picture of the
eigenvalues distribution. The Pellet theorem gives about {\tt 1.65e-4}
and {\tt 5.34} as lower and upper bound to the moduli of the
eigenvalues, but the tropical roots are rather small and near to the
lower bound. More precisely, the tropical
roots are {\tt 1.4e-3} and {\tt 1.7e-4} with multiplicities $3$ and $1$,
respectively.

\begin{figure}[h!t]
  \begin{minipage}{0.49\textwidth}  
    \includegraphics[width=\textwidth]{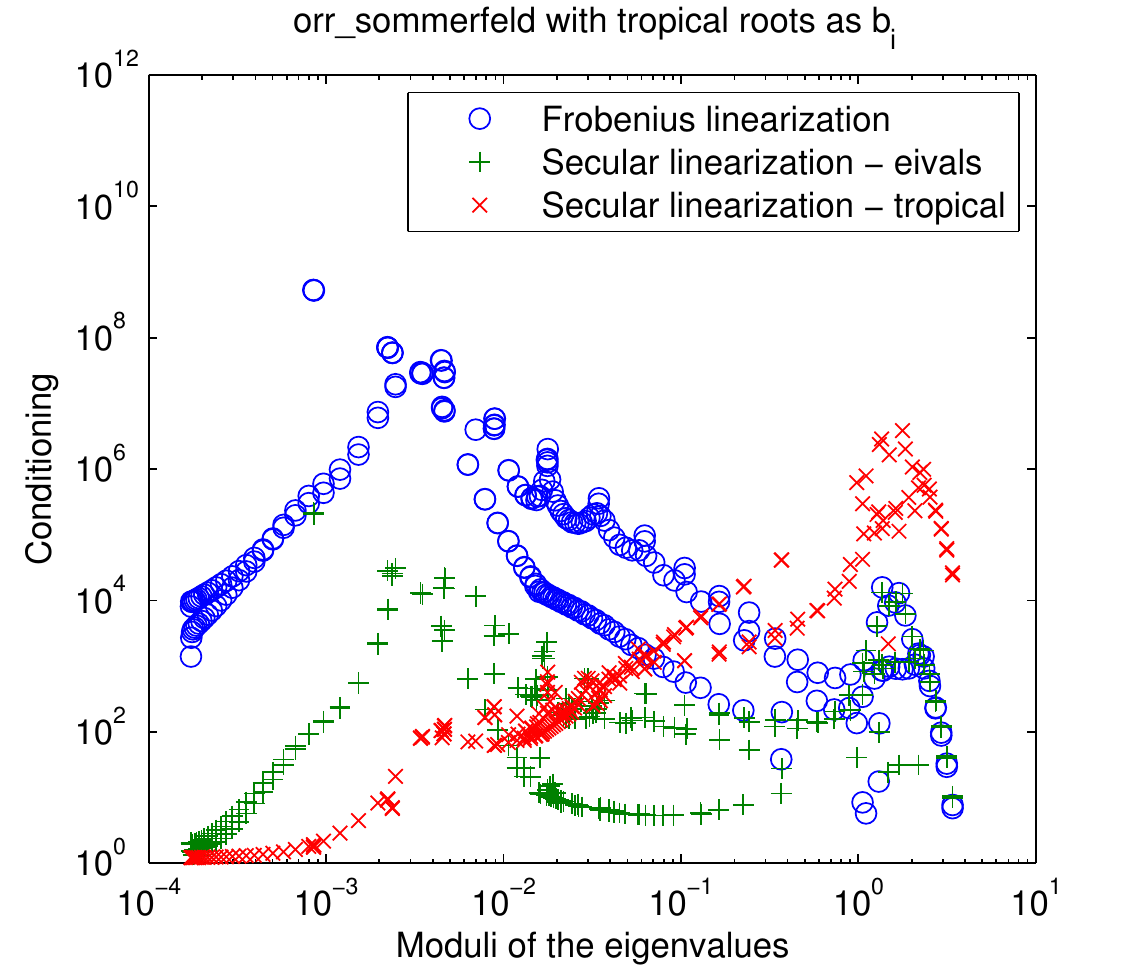}
  \end{minipage}
  \begin{minipage}{0.49\textwidth}
    \includegraphics[width=\textwidth]{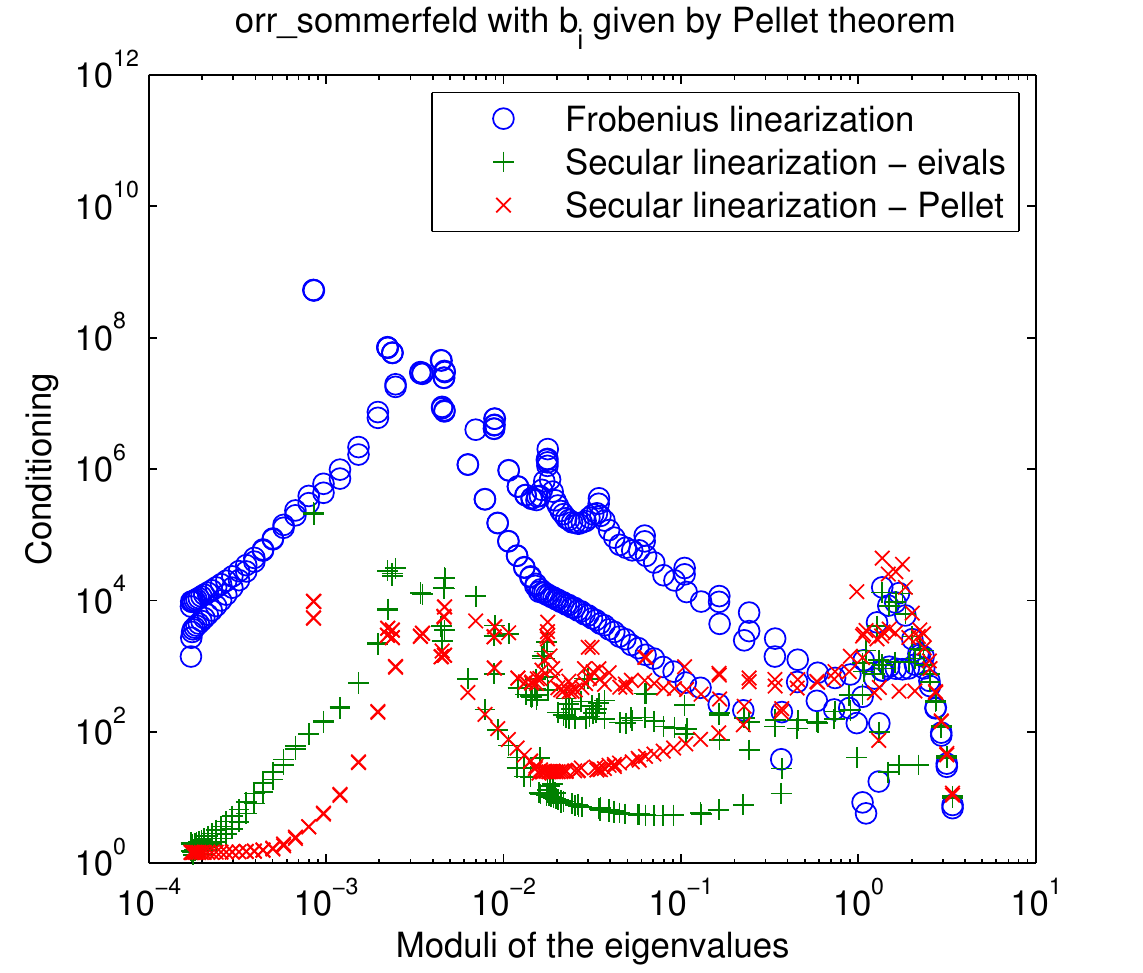}
  \end{minipage}
  
  \caption{On the left we report the conditioning of the Frobenius and of the secular linearization
  with the choices of $b_i$ as mean of subsets of eigenvalues with
  close moduli
  and as the estimates given by the tropical roots. On the right the tropical roots
  are coupled with estimates given by the Pellet theorem. }
  \label{fig:orr_sommerfeld}
\end{figure}

\begin{figure}[ht!] 
  \begin{center}
    \includegraphics[width=.7\textwidth]{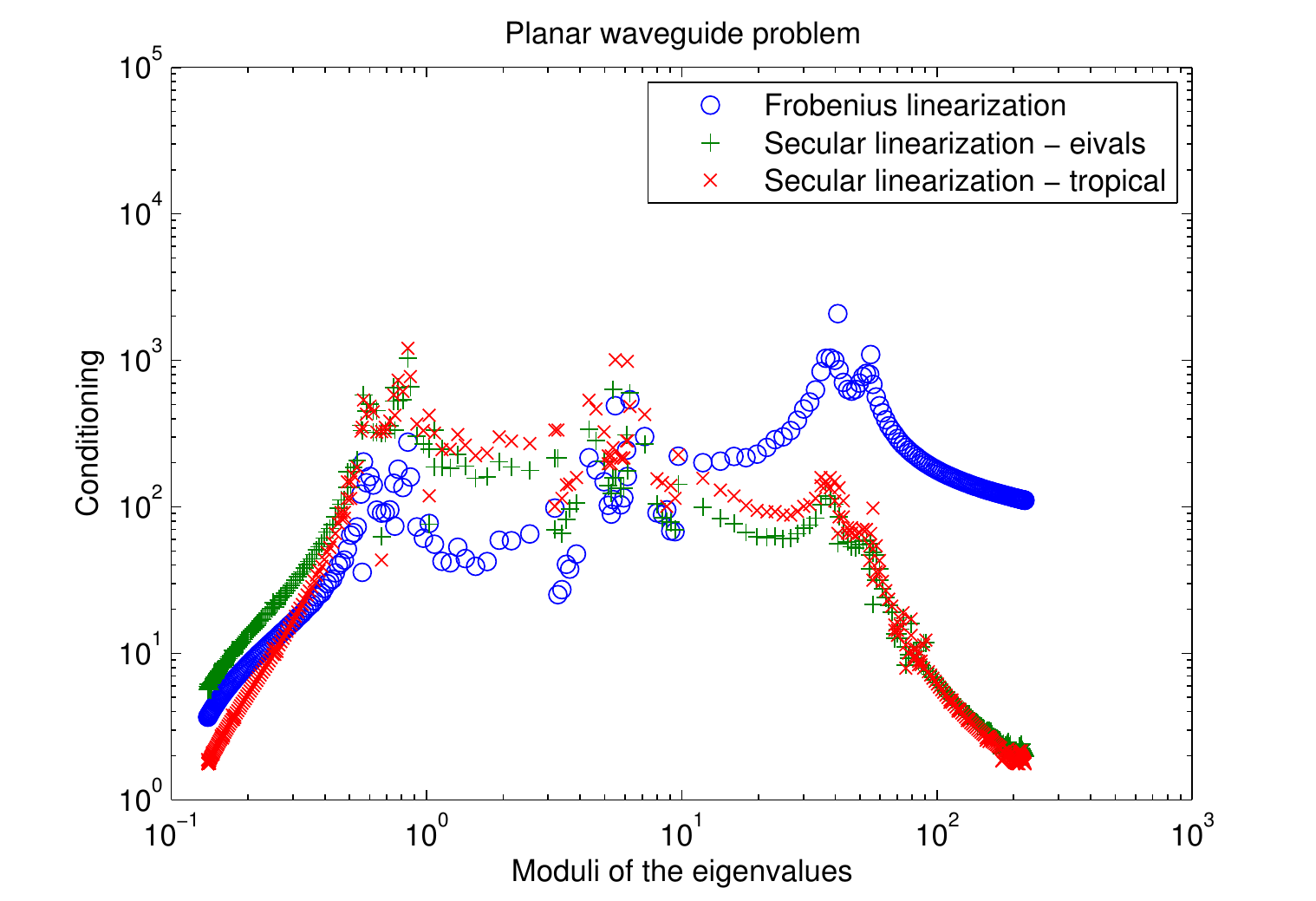}
  \end{center}
  
  \caption{Conditioning of the eigenvalues for three different linearizations 
    on the {\tt planar\_waveguide} problem. }
\label{fig:planar_waveguide}    
\end{figure}

This leads to a linearization $A(x)$ that is well-conditioned for the
smaller eigenvalues but with a higher conditioning on the eigenvalues
of bigger modulus as can be seen in Figure~\ref{fig:orr_sommerfeld} on
the left (the eigenvalues are ordered in nonincreasing order with
respect to their modulus). It can be seen, though, that coupling the
tropical roots with the standard Pellet theorem and altering the $b_i$
by adding a value slightly smaller than the upper bound (in this
example we have chosen $5$ but the result is not very sensitive to
this choice) leads to a much better result that is reported in
Figure~\ref{fig:orr_sommerfeld} on the right. In the right figure we
have used {\tt b = [ 1.7e-4, 1.4e-3, -1.4e-3, 5 ]}.  This seems to
justify that there exists a link between the quality of the
approximations obtained through the tropical roots and the
conditioning properties of the secular linearization.

We analyzed another example problem from the NLEVP collection
that is called {\tt planar\_waveguide}. The results are shown 
in Figure~\ref{fig:planar_waveguide}. 
This problem is a PEP of degree $4$ with two tropical roots approximately equal to 
$127.9$ and $1.24$. Again, it can be seen that for the eigenvalues of smaller modulus
(that will be near the tropical root $1.24$) the Frobenius linearization and the
secular one behave in the same way, whilst for the bigger ones the secular linearization
has some advantage in the conditioning. This may be justified by the fact that the
Frobenius linearization is similar to a secular linearization on the roots of the unity. 

Note that in this case the information obtained by the tropical roots seems more
accurate than in the {\tt orr\_sommerfeld} case, so the secular linearization built
using the tropical roots and the one built using the block-mean of the eigenvalues 
behave approximately in the same way. 

\begin{figure}[ht!]
\begin{center}
  \includegraphics[width=.7\textwidth]{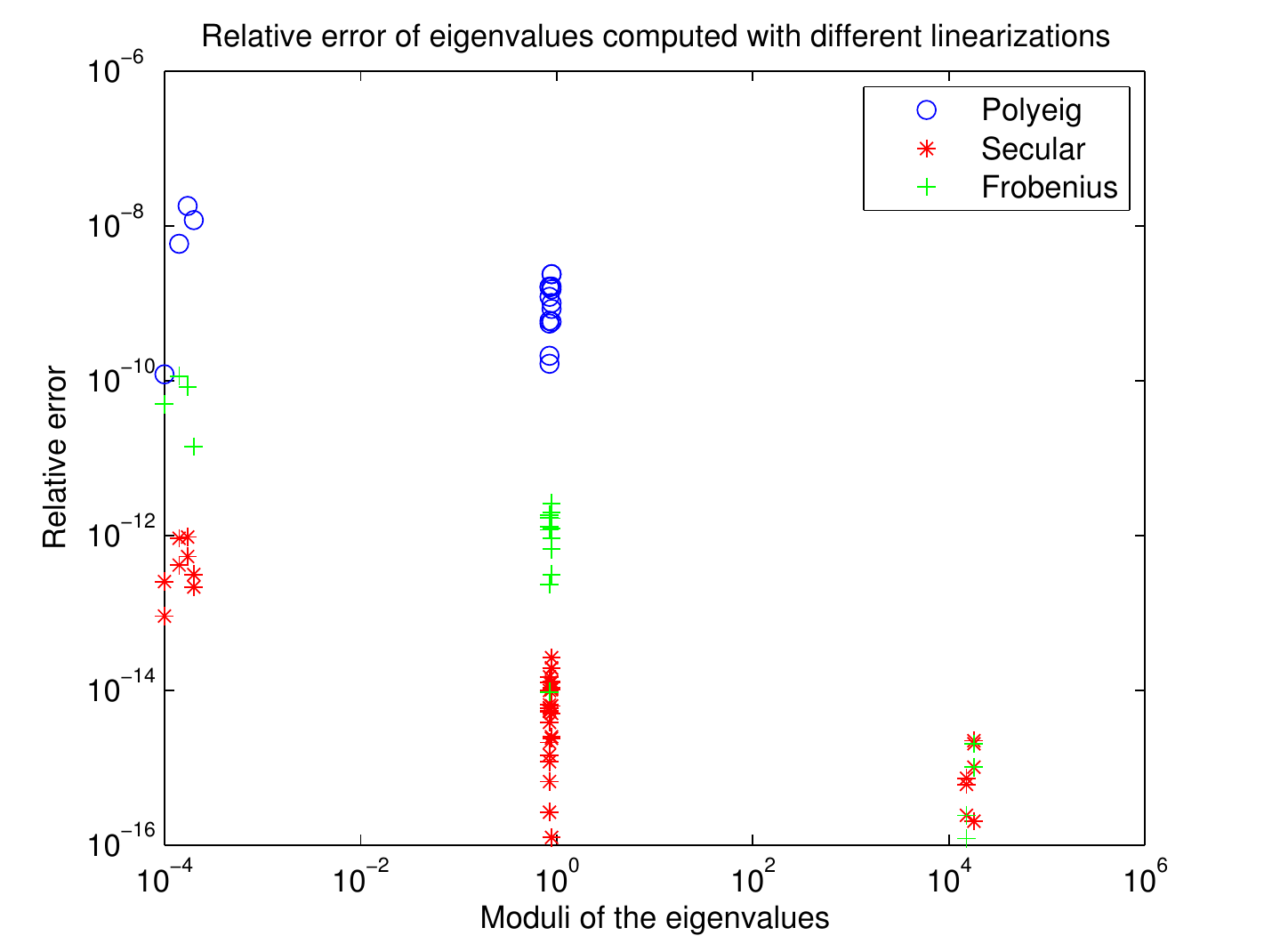}
\end{center}
\caption{The accuracy of the computed eigenvalues using polyeig, the Frobenius linearization and the secular
linearization with the $b_i$ obtained through the computation of the tropical roots.}
\label{fig:esempio_zeri}
\end{figure}

As a last example, we have tried to find the eigenvalues of a matrix polynomial
defined by integer coefficients. We have used polyeig and our secular linearization
(using the tropical roots as $b_i$) and the QZ method. We have chosen the polynomial
$P(x) = P_{11} x^{11} + P_9 x^9 + P_2 x^2 + P_0$ where
{\small \[
  P_{11} = \begin{bmatrix} 
    1 & 1 & 1 & 1 \\
      & 1 & 1 & 1 \\
      &   & 1 & 1 \\
      &   &   & 1 \\
  \end{bmatrix}, ~
  P_{9} = 10^8 \begin{bmatrix} 
    3 & 1 \\
    1 & 3 & 1 \\
      & 1 & 3 & 1 \\
      &   & 1 & 3 \\
  \end{bmatrix},
  P_2 = 10^8 P_{11}^t,~
  P_0 = \begin{bmatrix}
    1 \\
    & 2 \\
    & & 3 \\
    & & & 4 \\
  \end{bmatrix}.
\]
}
In this case the tropical roots are good estimates of the blocks of eigenvalues of
the matrix polynomial. We obtain the tropical roots $1.2664\cdot 10^4$, $0.9347$
and $1.1786\cdot 10^{-4}$ with multiplicities $2$, $7$ and $2$, respectively. We
have computed the eigenvalues with a higher precision and 
we have compared them with the results
of {\tt polyeig} and of {\tt eig} applied to the secular linearization and to the standard Frobenius linearization. 
Here, the secular linearization has been computed with the standard floating point arithmetic. As shown in Figure~\ref{fig:esempio_zeri} we have achieved much better accuracy with the latter choice. The secular linearization has achieved a relative error of the order
of the machine precision on all the eigenvalues except the smaller block (with modulus
about $10^{-4}$). In that case the relative error is about $10^{-12}$ but the absolute
error is, again, of the order of the machine precision. 
Moreover, {\tt polyeig} fails to detect the eigenvalues with
bigger modules, and marks them as eigenvalues at infinity. 
This can be noted by the fact that the circles relative
to the bigger eigenvalues are missing in {\tt polyeig}
plot of Figure~\ref{fig:esempio_zeri}. 

\section*{Acknowledgments}
We wish to thank the referees for their careful comments and useful remarks that
helped us to improve the presentation.

\bibliographystyle{plain}
\bibliography{biblio}

\end{document}